\documentclass[11pt,a4paper]{amsart}
\usepackage[T1]{fontenc}
\usepackage[utf8]{inputenc}
\usepackage[english]{babel}
\usepackage{microtype}
\usepackage{csquotes}
\usepackage[a4paper,textwidth=145mm,textheight=230mm,centering]{geometry}
\usepackage{amssymb,amsmath,amsthm}
\usepackage{lmodern,amsmath,amssymb,mathtools,microtype,tikz}
\usepackage[hidelinks]{hyperref}
\newtheorem{theorem}{Theorem}
\newtheorem{proposition}[theorem]{Proposition}
\newtheorem{lemma}[theorem]{Lemma}
\newcommand{\Rd}{\mathbb R^d}
\newcommand{\A}{\mathcal A}
\newcommand{\Ml}{\mathcal M_{\mathrm{lac}}}
\newcommand{\Mf}{\mathcal M_{\mathrm{full}}}
\newcommand{\M}{\mathcal M}
\newcommand{\Sfull}{S_{\mathrm{full}}}
\title[Strong-type endpoint obstructions]{Strong-type endpoint obstructions for bilinear spherical maximal operators}
\author{Vicente Vergara}
\address{Departamento de Matem\'atica, Facultad de Ciencias F\'isicas y Matem\'aticas, Universidad de Concepci\'on, Concepci\'on, Chile}
\email{vvergaraa@udec.cl}
\begin{document}
\begin{abstract}
In dimensions $d\ge2$, the individual bilinear spherical averages satisfy the H\"older bounds throughout the range $1\le p,q\le\infty$. We show that passing to the maximal operator destroys the strong bounds whenever one input exponent equals one, even when the supremum is restricted to dyadic radii. Together with the known interior estimates, this determines the exact strong-type boundedness regions of the full and lacunary operators. The obstruction is realized by accumulation on disjoint spatial shells: for $N$ consecutive dyadic radii, the $L^1\times L^q\to L^r$ operator norm, with $1/r=1+1/q$, grows at least linearly in $N$, and the growth is exactly of order $N$ for $d/(d-1)<q\le\infty$. In the range $d/(d-1)<q<\infty$, we characterize the surviving Lorentz bounds: either maximal operator maps $L^1\times L^{q,s}$ to $L^{r,t}$ if and only if $s\le t$. The strong-type obstruction extends to multilinear spherical maximal operators with an $L^1$ input.
\end{abstract}
\maketitle

\section{Introduction and main results}
Uniform bounds for a family of averaging operators need not persist when one takes a supremum over scales. Spherical means provide a fundamental instance of this distinction: the individual linear averages are contractions on $L^p$ for $1\le p\le\infty$, whereas the full spherical maximal operator is bounded on $L^p(\Rd)$ precisely for $p>d/(d-1)$, by Stein \cite{Stein} in dimensions $d\ge3$ and Bourgain \cite{Bourgain} in dimension two. Pointwise convergence of spherical means is one motivation for these maximal functions. Restriction to dyadic radii improves the range to $1<p\le\infty$ for $d\ge2$, by Calder\'on \cite{Calderon} and Coifman--Weiss \cite{CW}. This article investigates the corresponding distinction at the $L^1$ boundary of the bilinear theory, where fixed-scale H\"older estimates hold but uniform control of the maximal output requires a different integrability condition.

For $d\ge2$, let $\sigma_{2d-1}$ be normalized surface measure on $S^{2d-1}\subset\Rd\times\Rd$. Define
\begin{equation}\label{eq:avg}
\A_u(f,g)(x)=\int_{S^{2d-1}}f(x-uy)g(x-uz)\,d\sigma_{2d-1}(y,z),\qquad u>0,
\end{equation}
and the associated maximal operators
\begin{equation}\label{eq:max}
\Ml(f,g)=\sup_{k\in\mathbb Z}|\A_{2^k}(f,g)|,
\qquad \Mf(f,g)=\sup_{u>0}|\A_u(f,g)|.
\end{equation}
Geba, Greenleaf, Iosevich, Palsson and Sawyer \cite{GGIPS} initiated the study of the bilinear spherical maximal operator, and further estimates were obtained in \cite{BGHHO,GHH,HHY}. The fixed-scale theory and the maximal theory exhibit a striking difference. In dimensions $d\ge2$, Iosevich--Palsson--Sovine \cite{IPS} established the single-scale bound $\A_1:L^1\times L^1\to L^{1/2}$. Interpolation with the elementary H\"older bounds gives
\begin{equation}\label{eq:fixed}
\sup_{u>0}\|\A_u(f,g)\|_{L^r}\lesssim_{d,p,q}\|f\|_{L^p}\|g\|_{L^q},
\qquad 1\le p,q\le\infty,\quad \frac1r=\frac1p+\frac1q,
\end{equation}
where dilation invariance gives uniformity in $u$; see also \cite{BF,Borges}. In particular, an $L^1$ input is admissible for every single radius.

For the full maximal operator, the slicing method of Jeong--Lee \cite{JL} yields strong estimates in the range $1<p,q\le\infty$ and $r>d/(2d-1)$. The latter strict inequality is necessary (see \cite{BGHHO} and the cap test in Section~\ref{sec:context}). For the lacunary operator, Borges--Foster \cite[Theorem~1]{BF} established strong estimates throughout $1<p,q\le\infty$ with $1/r=1/p+1/q$. Related advances concern sparse domination \cite{BFOPZ}, other dilation sets \cite{BFO,CSS}, and restricted weak-type boundary estimates \cite{BCSS}. The survey \cite{Borges} describes the state of these theories and formulates several remaining boundary questions.

The H\"older relation is necessary for the full and lacunary maximal operators by their respective continuous and dyadic dilation invariances. The difference between \eqref{eq:fixed} and the maximal estimates thus lies not in the scaling relation but in the range of integrability for which that relation can be realized. Does the improvement under lacunarity extend to the $L^1$ boundary? In particular, Borges \cite[Question~3.2(1)]{Borges} asks whether $\Ml:L^1\times L^1\to L^{1/2}$, or $\Ml:L^1\times L^{1+\varepsilon}\to L^{(1+\varepsilon)/(2+\varepsilon)}$ for small $\varepsilon>0$, is bounded. Earlier, Borges--Foster \cite[Remark~2]{BF} had conjectured strong estimates along the face $p=1$, $1<q<\infty$, while reporting bounds inherited from the full operator. The survey also records subsequent asserted positive estimates in Section~8. Our first theorem resolves the strong-type question on the entire face. Its proof gives a geometric explanation for why the uniform bounds of \eqref{eq:fixed} do not persist under the supremum.

\begin{theorem}\label{thm:main}
For every $d\ge2$ and $1\le q\le\infty$, set $1/r=1+1/q$, with $1/\infty=0$. There exist nonnegative functions $f\in C_c^\infty(\Rd)$ and $g\in C^\infty(\Rd)\cap L^q(\Rd)$ such that
\[
\Ml(f,g)\notin L^r(\Rd).
\]
Consequently neither $\Ml$ nor $\Mf$ is bounded from $L^1\times L^q$ into $L^r$. The same holds with the two inputs interchanged.
\end{theorem}

Combining Theorem~\ref{thm:main} with the known interior estimates and the necessary condition for the full operator gives the following exact description. Under $1/r=1/p+1/q$, $1\le p,q\le\infty$, the strong-type boundedness regions are
\begin{equation}\label{eq:regions}
\begin{array}{ll}
\Ml:&1<p,q\le\infty,\\[2pt]
\Mf:&1<p,q\le\infty,\quad r>\dfrac{d}{2d-1}.
\end{array}
\end{equation}
The origin in reciprocal-exponent coordinates ($p=q=\infty$) is included. Thus the improvement in the strong range afforded by lacunarity does not extend to any point with an $L^1$ input. Figure~\ref{fig:regions} summarizes the two regions.

\begin{figure}[ht]
\centering
\begin{tikzpicture}[scale=2.8, every node/.style={font=\footnotesize}]
\definecolor{region}{RGB}{221,232,239}
\definecolor{newedge}{RGB}{155,53,43}
\definecolor{oldedge}{RGB}{89,96,106}
\foreach \shift/\kind in {0/full,1.6/lac} {
\begin{scope}[xshift=\shift cm]
\ifnum\pdfstrcmp{\kind}{full}=0
\fill[region] (0,0)--(1,0)--(1,.6667)--(.6667,1)--(0,1)--cycle;
\draw[oldedge,dashed,thick] (1,.6667)--(.6667,1);
\draw[newedge,dashed,very thick] (1,0)--(1,.6667) (0,1)--(.6667,1);
\draw[oldedge,fill=white] (1,.6667) circle (.022) (.6667,1) circle (.022);
\node[below] at (.6667,0) {$a$}; \draw (.6667,0)--(.6667,.025);
\node[left] at (0,.6667) {$a$}; \draw (0,.6667)--(.025,.6667);
\node at (.5,1.22) {$\Mf$};
\else
\fill[region] (0,0) rectangle (1,1);
\draw[newedge,dashed,very thick] (1,0)--(1,1)--(0,1);
\draw[newedge,fill=white,thick] (1,1) circle (.022);
\node at (.5,1.22) {$\Ml$};
\fi
\draw[->] (0,0)--(1.15,0) node[right] {$1/p$};
\draw[->] (0,0)--(0,1.12) node[above] {$1/q$};
\draw[very thick] (0,0)--(1,0) (0,0)--(0,1);
\fill (0,0) circle (.016);
\draw[oldedge,fill=white,thick] (1,0) circle (.022) (0,1) circle (.022);
\node[below] at (1,0) {$1$}; \node[left] at (0,1) {$1$};
\end{scope}}
\end{tikzpicture}
\caption{Strong-type regions: shaded area and black axes, excluding dashed edges and open circles. The full panel shows $d=3$; in general the diagonal joins $(a,1)$ and $(1,a)$, $a=(d-1)/d$. The lacunary panel holds for all $d\ge2$. Red dashes mark the $L^1$ obstruction; the gray diagonal and open axis endpoints were already known exclusions.}\label{fig:regions}
\end{figure}

The mechanism behind Theorem~\ref{thm:main} is spatial accumulation across scales. One input is concentrated near the origin, while the other has a nearly critical radial tail. At radius $2^k$ the first input contributes a factor comparable to $2^{-kd}$ on a shell of measure comparable to $2^{kd}$. The second input contributes a factor $2^{-kd/q}$, up to a logarithmic weight. The H\"older relation exactly balances the powers of $2^k$ in the $r$th-power integral, but it does not ensure summability over the disjoint shells. This explains why restricting to dyadic radii does not remove the obstruction.

To quantify the loss caused by the supremum, we consider the norm as a function of the number of scales. For $N\ge1$ write
\[
\M_N(f,g)=\max_{1\le k\le N}|\A_{2^k}(f,g)|.
\]
\begin{proposition}[Finite-scale growth]\label{prop:finite}
Let $d\ge2$, $1\le q\le\infty$ and $1/r=1+1/q$. Then
\[
\|\M_N\|_{L^1\times L^q\to L^r}\gtrsim_{d,q} N,
\]
with nonnegative smooth compactly supported witnesses. If $d/(d-1)<q\le\infty$, the matching upper bound holds for any $N$ positive radii, so
\begin{equation}\label{eq:growth}
\|\M_N\|_{L^1\times L^q\to L^r}\asymp_{d,q}N.
\end{equation}
\end{proposition}

For comparison, \eqref{eq:fixed} gives a bound independent of the chosen single radius. Proposition~\ref{prop:finite} shows that the loss from taking the supremum over consecutive dyadic scales is exactly linear in the noncritical range. The question is then which refinements of the input and output spaces restore a uniform estimate on the same face. On the portion where the linear full spherical maximal operator has strong $L^q$ bounds, the following theorem provides an exact answer in terms of Lorentz indices.

\begin{theorem}[Sharp Lorentz indices]\label{thm:lorentz}
Let $d\ge2$, $d/(d-1)<q<\infty$ and $1/r=1+1/q$. For $0<s,t\le\infty$, each of the maps
\[
\Mf,\ \Ml:\quad L^1(\Rd)\times L^{q,s}(\Rd)\longrightarrow L^{r,t}(\Rd)
\]
is bounded if and only if $s\le t$.
\end{theorem}
In particular, an $L^q=L^{q,q}$ second input gives output in $L^{r,q}$ and this secondary index cannot be decreased. The stronger output $L^r=L^{r,r}$ is available within the Lorentz scale precisely for second inputs in $L^{q,s}$ with $s\le r$. The logarithmic factors in the shell construction measure this distinction: the input and output sum the same tail with different secondary indices. The strong estimate fails since $r<q$. At $q=\infty$ one has the separate weak estimate $L^1\times L^\infty\to L^{1,\infty}$. Theorem~\ref{thm:lorentz} makes no positive assertion at or below the critical value $q=d/(d-1)$.

The same shell argument extends to multilinear spherical maximal operators whenever at least one input exponent equals one (Proposition~\ref{prop:multi}). The argument for the strong obstruction uses only nonnegative functions and therefore applies also to maximal operators with the absolute value taken inside the average. It does not by itself settle weak-type estimates. In subsequent work on the vertex $p=q=1$, Lee--Li--Wen--Wu \cite{LLWW}, citing the first version of the present paper for the strong obstruction in $d\ge2$, establish that even the $H^1\times H^1\to L^{1/2,\infty}$ estimate for the lacunary operator fails in every dimension $d\ge1$. In particular, the weak $L^1\times L^1$ endpoint fails, for both lacunary and full maximal operators. They also obtain logarithmic lower growth for a weak-output norm with Hardy-space inputs on suitable finite sets of dyadic radii. This concerns a different norm and endpoint from Proposition~\ref{prop:finite}; the failure of weak type at the vertex does not follow from our strong-type shell estimate. The positive Lorentz estimates on the noncritical face are unaffected. Positive restricted weak-type estimates at other boundary points are treated in \cite{JL,BCSS}.

\section{The endpoint construction and finite-scale growth}\label{sec:construction}
For every nonnegative measurable $H$ on $S^{2d-1}$, the normalized spherical disintegration formula is
\begin{equation}\label{eq:slice}
\begin{aligned}
\int_{S^{2d-1}}H(y,z)\,d\sigma_{2d-1}(y,z)
={}&c_d\int_{B(0,1)}(1-|y|^2)^{(d-2)/2}\\[-2pt]
&\times\int_{S^{d-1}}H(y,\sqrt{1-|y|^2}\,\omega)\,d\sigma_{d-1}(\omega)\,dy,
\end{aligned}
\end{equation}
where both measures on spheres are normalized and
\[
c_d=\frac{\Gamma(d)}{\pi^{d/2}\Gamma(d/2)}.
\]
The constant follows from the beta integral, since the integral of the marginal weight over the unit ball is $\pi^{d/2}\Gamma(d/2)/\Gamma(d)$. In particular the weight is bounded from below on every closed subball of the unit ball, also for $d=2$.

\begin{proof}[Proof of Theorem~\ref{thm:main}]
Choose a nonnegative $f\in C_c^\infty(\Rd)$ supported in $B(0,1)$, with $f\ge1$ on $B(0,1/2)$. Let $\eta\in C^\infty([0,\infty))$ satisfy $0\le\eta\le1$, $\eta=0$ on $[0,3]$ and $\eta=1$ on $[4,\infty)$. For $q<\infty$ select $\alpha$ so that
\begin{equation}\label{eq:alpha}
1/q<\alpha\le1/r=1+1/q,
\end{equation}
and put
\begin{equation}\label{eq:tail}
g_\alpha(u)=\eta(|u|)|u|^{-d/q}(\log|u|)^{-\alpha},
\end{equation}
extended by zero near the origin. This is smooth and belongs to $L^q$ since $\alpha q>1$:
\[
\int_{\Rd}|g_\alpha|^q\lesssim 1+\int_4^\infty\frac{d\rho}{\rho(\log\rho)^{\alpha q}}<\infty.
\]
For $q=\infty$ take $g\equiv1$ and $\alpha=0$.

For $k\ge6$, write $u=2^k$ and define
\begin{equation}\label{eq:shell}
D_k=\{x:u/2\le |x|\le3u/5\},\qquad
Y_{u,x}=\{y:|y-x/u|\le1/(2u)\}.
\end{equation}
If $x\in D_k$ and $y\in Y_{u,x}$, then $f(x-uy)\ge1$, $|y|\le 3/5+1/(2u)<5/8$ and $\sqrt{1-|y|^2}>3/4$. Thus for every $\omega\in S^{d-1}$,
\begin{equation}\label{eq:arguments}
3u/20\le|x-u\sqrt{1-|y|^2}\,\omega|\le8u/5.
\end{equation}
The cutoff in \eqref{eq:tail} equals one here because $u\ge64$. Consequently, for finite $q$ the radial tail is bounded below by a constant times $u^{-d/q}k^{-\alpha}$; the same holds for the constant second input with $q=\infty$. On $Y_{u,x}$ the weight in \eqref{eq:slice} is at least $(39/64)^{(d-2)/2}$ and the ball has measure $|B(0,1)|(2u)^{-d}$. Restricting the integration to this ball gives
\begin{equation}\label{eq:lower}
\Ml(f,g)(x)\ge\A_{2^k}(f,g)(x)\gtrsim 2^{-kd/r}k^{-\alpha},\qquad x\in D_k.
\end{equation}
The shells are pairwise disjoint, with $|D_k|=c'_d2^{kd}$ for the positive constant $c'_d=|B(0,1)|((3/5)^d-(1/2)^d)$. Therefore
\[
\int_{\Rd}\Ml(f,g)^r\,dx\gtrsim \sum_{k\ge6} k^{-\alpha r}=\infty
\]
by \eqref{eq:alpha}, also when $q=\infty$ and $\alpha=0$. Symmetry and $\Mf\ge\Ml$ give the remaining assertions.
\end{proof}

\begin{proof}[Proof of Proposition~\ref{prop:finite}]
First consider the $N$ radii $2^6,\ldots,2^{N+5}$. Keep the function $f$ from the preceding proof. For $K=N+5$, choose a smooth radial cutoff $0\le\chi_K\le1$ equal to one on $4\le|x|\le2^{K+1}$ and supported in $2<|x|<2^{K+2}$. Set $g_K(x)=\chi_K(x)|x|^{-d/q}$, extended by zero near the origin (with the power zero for $q=\infty$). Both functions are nonnegative and smooth with compact support. For finite $q$, polar integration gives
\[
\|g_K\|_q^q\lesssim\int_2^{2^{K+2}}\frac{d\rho}{\rho}\lesssim N,
\]
whereas $\|g_K\|_\infty\le1$. The arguments in \eqref{eq:arguments} lie where $\chi_K=1$ for $6\le k\le K$. The proof of \eqref{eq:lower} with $\alpha=0$ thus yields an output whose $r$th-power integral is bounded below by $c_{d,q}N$. The norm ratio is bounded below by a constant times $N^{1/r-1/q}=N$. The H\"older relation makes the operator norm invariant under a common dilation of all radii: for $D_a h(x)=h(ax)$, $\A_u(D_af,D_ag)(x)=\A_{au}(f,g)(ax)$. Taking $a=2^5$ transfers the lower estimate and the smooth witnesses to the radii $2^1,\ldots,2^N$.

For the upper estimate let $E\subset(0,\infty)$ consist of $N$ radii and put
\[
B_E f(x)=\max_{u\in E}\frac{1}{|B(0,u)|}\int_{B(0,u)}|f(x-v)|\,dv.
\]
Writing $\Sfull g$ for the linear full spherical maximal function applied to $|g|$, formula \eqref{eq:slice}, with $(1-|y|^2)^{(d-2)/2}\le1$, yields
\begin{equation}\label{eq:truncdom}
\max_{u\in E}|\A_u(f,g)(x)|\le c_d|B(0,1)|B_E f(x)\Sfull g(x).
\end{equation}
The integral of each normalized ball average is $\|f\|_1$, so $\|B_E f\|_1\le N\|f\|_1$. By H\"older's inequality applied to the $r$th power (or directly for $q=\infty$, $r=1$),
\[
\left\|\max_{u\in E}|\A_u(f,g)|\right\|_r
\le C_d\|B_Ef\|_1\|\Sfull g\|_q
\le C_{d,q}N\|f\|_1\|g\|_q.
\]
The last inequality is the linear spherical maximal bound of Stein--Bourgain \cite{Stein,Bourgain} for $q>d/(d-1)$, or the immediate supremum bound at $q=\infty$. This proves \eqref{eq:growth}.
\end{proof}

\section{Lorentz bounds on the noncritical face}
For $0<p<\infty$ and $0<s\le\infty$, we use the rearrangement convention
\[
\|h\|_{L^{p,s}}=\left(\int_0^\infty[v^{1/p}h^*(v)]^s\,\frac{dv}{v}\right)^{1/s},
\quad \|h\|_{L^{p,\infty}}=\sup_{v>0}v^{1/p}h^*(v),
\]
where the first formula is for $s<\infty$ and $h^*$ is the decreasing rearrangement of $|h|$. These are quasi-norms in the relevant non-Banach ranges. We use the standard nesting $L^{p,s}\hookrightarrow L^{p,t}$ for $s\le t$ and real interpolation of Lebesgue spaces; see \cite{BS}.

\begin{lemma}[The Lorentz tail test]\label{lem:tail}
Let $1\le q<\infty$, $1/r=1+1/q$ and $\alpha\ge0$. For the function in \eqref{eq:tail},
\[
g_\alpha\in L^{q,s}\ \Longleftrightarrow\ \alpha s>1\quad(0<s<\infty),\qquad
g_\alpha\in L^{q,\infty}.
\]
For the fixed $f$ above, $\Ml(f,g_\alpha)\notin L^{r,t}$ if $0<t<\infty$ and $\alpha t\le1$.
\end{lemma}
\begin{proof}
For large $v$ one has $g_\alpha^*(v)\asymp v^{-1/q}(\log v)^{-\alpha}$ by radial monotonicity outside a fixed ball. Boundedness of $g_\alpha$ controls the integral at small $v$, so the first assertions follow. The geometric proof of \eqref{eq:lower} is valid for all $\alpha\ge0$, independently of $L^q$ membership. For $H=\Ml(f,g_\alpha)$ and $m_k=|D_k|=c'_d2^{kd}$, it implies $H^*(m_k/2)\gtrsim 2^{-kd/r}k^{-\alpha}$. Choose $a,b$ with $2^{-d}b<a<b<1/2$. The intervals $[am_k,bm_k]$ are disjoint. On each of them the same lower bound follows by monotonicity, and therefore
\[
\int_0^\infty[v^{1/r}H^*(v)]^t\frac{dv}{v}\gtrsim\sum_{k\ge6}k^{-\alpha t}.
\]
This diverges when $\alpha t\le1$.
\end{proof}

\begin{proof}[Proof of Theorem~\ref{thm:lorentz}]
By taking the supremum over all radii in \eqref{eq:truncdom}, we recover the slicing estimate of Jeong--Lee \cite[Lemma~2.1]{JL}:
\begin{equation}\label{eq:dom}
\Mf(f,g)(x)\lesssim M_{\mathrm{HL}}f(x)\Sfull g(x).
\end{equation}
Here $M_{\mathrm{HL}}$ is the centered Hardy--Littlewood maximal operator, of weak type $(1,1)$. In the stated range of $q$, real interpolation between nearby strong Lebesgue bounds shows that $\Sfull$ is bounded on $L^{q,s}$ for every $0<s\le\infty$. The elementary rearrangement inequality $(ab)^*(2v)\le a^*(v)b^*(v)$ gives
\[
\|ab\|_{L^{r,s}}\lesssim\|a\|_{L^{1,\infty}}\|b\|_{L^{q,s}},\qquad 0<s\le\infty.
\]
Thus \eqref{eq:dom} and Lorentz nesting imply both estimates for $s\le t$. Conversely, for $t<s<\infty$ choose $1/s<\alpha\le1/t$. Lemma~\ref{lem:tail} gives $g_\alpha\in L^{q,s}$ and failure of the lacunary output to belong to $L^{r,t}$. For $s=\infty$ and finite $t$ take $0\le\alpha\le1/t$. Domination gives the full maximal failure as well.
\end{proof}
The product estimate preserves the secondary index of the second input; accumulation on the disjoint shells prevents any improvement in that index.

\section{A multilinear extension}
For $m\ge2$, $d\ge2$, let $\sigma_{md-1}$ denote normalized surface measure on $S^{md-1}\subset(\Rd)^m$ and put
\[
\A_u^{(m)}(f_1,\ldots,f_m)(x)=\int_{S^{md-1}}\prod_{j=1}^m f_j(x-uy_j)\,d\sigma_{md-1}(y),
\]
with $\M_{m,\mathrm{lac}}=\sup_{k\in\mathbb Z}|\A_{2^k}^{(m)}|$ and $\M_{m,\mathrm{full}}=\sup_{u>0}|\A_u^{(m)}|$.
\begin{proposition}\label{prop:multi}
If $1\le q_j\le\infty$ for $2\le j\le m$ and $1/r=1+\sum_{j=2}^m1/q_j$, there exist nonnegative $f_1\in C_c^\infty(\Rd)$ and $f_j\in C^\infty(\Rd)\cap L^{q_j}(\Rd)$ such that $\M_{m,\mathrm{lac}}(f_1,\ldots,f_m)\notin L^r(\Rd)$. In particular every strong-type H\"older estimate with an $L^1$ input fails for both maximal operators.
\end{proposition}
\begin{proof}
Disintegration with respect to the first $d$ coordinates gives
\begin{align*}
\int_{S^{md-1}}F(y_1,y')\,d\sigma_{md-1}
={}&C_{m,d}\int_{B(0,1)}(1-|y_1|^2)^{((m-1)d-2)/2}\\[-2pt]
&\times\int_{S^{(m-1)d-1}}F(y_1,\sqrt{1-|y_1|^2}\,\omega)\,d\sigma(\omega)\,dy_1,
\end{align*}
where $y'=(y_2,\ldots,y_m)$, $\omega=(\omega_2,\ldots,\omega_m)$ and $C_{m,d}=\Gamma(md/2)/(\pi^{d/2}\Gamma((m-1)d/2))$. The constant follows as in \eqref{eq:slice}.
Fix $f_1=f$ from Section~\ref{sec:construction}, and put $P=\sum_{j=2}^m1/q_j$. For each finite $q_j$ choose $\alpha_j>1/q_j$ such that
\[
A:=\sum_{j:q_j<\infty}\alpha_j<1+P=1/r.
\]
This is possible by taking the sum of the positive excesses smaller than one. Let $f_j(v)=\eta(|v|)|v|^{-d/q_j}(\log|v|)^{-\alpha_j}$, extended by zero near the origin; for $q_j=\infty$ use $f_j\equiv1$. These inputs belong to their stated spaces.

Fix $x\in D_k$, $u=2^k$ and restrict $y_1$ to $Y_{u,x}$ from \eqref{eq:shell}, where the marginal weight is bounded below and the volume is comparable to $u^{-d}$. On the residual sphere restrict further to $\omega_j\cdot x\le0$ for every $2\le j\le m$. Independent block reflections preserve surface measure and the boundaries have measure zero, so this set has normalized measure $2^{-(m-1)}$. Writing $b=\sqrt{1-|y_1|^2}\le1$, we have
\[
u/2\le|x-ub\omega_j|\le8u/5\qquad(2\le j\le m).
\]
The cutoffs are therefore one and the product of the remaining inputs is bounded from below by a constant times $u^{-dP}k^{-A}$. Thus the lacunary maximal output is at least $c2^{-kd/r}k^{-A}$ on $D_k$. Its $r$th-power integral dominates $\sum_{k\ge6}k^{-Ar}=\infty$. Symmetry and pointwise domination conclude the proof.
\end{proof}

\section{Strong-type regions and relation to earlier results}\label{sec:context}
For completeness, the remaining necessary condition for the full operator in \eqref{eq:regions} follows from a cap test. Choose nonnegative smooth compactly supported $f,g$ that are at least one near zero. For large $|x|$ set $u=\sqrt2|x|$. The point $(x/u,x/u)$ lies on $S^{2d-1}$. A cap about this point of radius $c/u$, with $c>0$ sufficiently small, keeps both arguments of $f$ and $g$ in their unit lower-bound regions. Its measure is comparable to $u^{-(2d-1)}$. Thus $\Mf(f,g)(x)\gtrsim |x|^{-(2d-1)}$, precluding strong $L^r$ membership if $r(2d-1)\le d$. The remaining exclusions are Theorem~\ref{thm:main}; the sufficient ranges are \cite[Section~2]{JL} and \cite[Theorem~1]{BF}, with $p=q=\infty$ immediate. This proves \eqref{eq:regions}.

We record how this description relates to earlier strong-type assertions. Theorem~1.1 of Jeong--Lee \cite[pp.~3--4]{JL} includes $p=1$ and finite $q>d/(d-1)$, for example $(d,p,q,r)=(3,1,2,2/3)$. Theorem~\ref{thm:main} excludes these estimates. The issue can be seen in the interpolation step immediately following (2.5) on p.~11 of \cite{JL}. For fixed $f\in L^1$ and a fixed measurable choice of radius $\tau$, let $T_f g=\A_{\tau}(f,g)$. The two bounds in (2.5), at distinct finite $q_0,q_1>d/(d-1)$, have the form
\[
\|T_fg\|_{L^{r_i,\infty}}\lesssim\|f\|_1\|g\|_{L^{q_i}},\qquad 1/r_i=1+1/q_i,\quad i=0,1,
\]
with constants independent of $\tau$. For $0<\theta<1$, put $1/q=(1-\theta)/q_0+\theta/q_1$ and $1/r=(1-\theta)/r_0+\theta/r_1$. Real interpolation at secondary index $r$ uses the identities
\[
(L^{q_0},L^{q_1})_{\theta,r}=L^{q,r},\qquad
(L^{r_0,\infty},L^{r_1,\infty})_{\theta,r}=L^{r,r}=L^r.
\]
Consequently it yields the uniform strong-output bound with $L^{q,r}$, not $L^q$, as the second input. Since $r<q$, the former is strictly smaller than the latter. Theorem~\ref{thm:lorentz} proves that this restriction is sharp for the maximal operator: a larger secondary index than $r$ cannot give a strong $L^r$ output. The weak estimates and the strong estimates with both input exponents greater than one in \cite{JL} are not contradicted. Indeed, the strong product estimates (2.4) of \cite{JL}, their symmetric counterparts and bilinear interpolation between points with both exponents greater than one suffice for the full range in \eqref{eq:regions}, without using the asserted strong bounds on $p=1$ or $q=1$.

The supplementary boundary claims in \cite[Remark~2]{BF} and \cite[Section~3]{Borges}, and the full-maximal ranges reproduced in \cite[Theorem~10]{BF} and \cite[Theorem~2.2]{Borges}, must accordingly omit these points. The main lacunary theorem of \cite{BF} is unaffected.

In \cite[Theorem~1.6, p.~1805]{CLS}, the specialization to $\Sigma=\tfrac12 S^{2d-1}$ and $\kappa=2d-1$ includes strong lacunary estimates for $p_1=1$, $2d/(2d-1)<p_2\le2$. The dyadic dilates of this sphere agree with those of the unit sphere after an index shift. Similarly, the specialization to dyadic radii of the results in \cite{CSS}, stated there as Theorem~2.3 (arXiv:2510.03094v3, 31 August 2026), asserts strong estimates for $p_1=1$, $1<p_2<\infty$, $d\ge4$. For instance, the assertion at $(d,p_1,p_2,r)=(4,1,2,2/3)$ is incompatible with Theorem~\ref{thm:main}. The boundary claims of \cite[Section~8]{Borges} include these assertions. This comparison identifies the incompatible inequalities; it does not locate a specific step in the proof of \cite{CSS} or address estimates for arbitrary dilation sets.

For the multilinear operator, \cite[Theorem~1, Case II(d)]{Dosidis}, in the 2019 preprint version, includes open strong-type faces with an $L^1$ input; Proposition~\ref{prop:multi} excludes them. This comparison concerns arXiv:1911.04071v1, not an independently verified final journal text. Earlier vertex counterexamples with exponents in $\{1,\infty\}$ are recalled in \cite[p.~1057]{DR}, with reference to \cite[Proposition~2]{Dosidis}; these are distinct from the obstruction for finite $1<q<\infty$ along the face. The full multilinear boundedness statement in \cite[Theorem~2]{Gao} is attributed to \cite{Dosidis} and includes, in particular, $(d,p_1,p_2,r)=(3,1,2,2/3)$ without an independent proof of this endpoint. The single-scale estimate in \cite[Theorem~1]{Gao} and its lacunary estimate with all inputs of exponent greater than one \cite[Theorem~4]{Gao} are unaffected.

These comparisons concern the specific strong maximal inequalities identified above. They do not assert failure of independent interior bounds, local or single-scale estimates, bounds for other surfaces, or weak-type estimates beyond the separate results cited in the introduction.

\end{document}